\documentclass[preprint,12pt]{article} 

\usepackage{amsmath,amssymb,amsthm}
\usepackage{mathtools}
\usepackage{microtype}
\usepackage{hyperref}
\usepackage{graphicx}
\usepackage{booktabs}
\usepackage{enumitem}
\usepackage{cleveref}
\usepackage{url}
\usepackage{accents}
\usepackage{stmaryrd}

\DeclareFontFamily{OMX}{yhex}{}
\DeclareFontShape{OMX}{yhex}{m}{n}{<->yhcmex10}{}
\DeclareSymbolFont{yhlargesymbols}{OMX}{yhex}{m}{n}
\DeclareMathAccent{\wideparen}{\mathord}{yhlargesymbols}{"F3}

\newtheorem{theorem}{Theorem}
\newtheorem{proposition}[theorem]{Proposition}
\newtheorem{lemma}[theorem]{Lemma}

\theoremstyle{definition}
\newtheorem{definition}[theorem]{Definition}

\theoremstyle{theorem}

\newtheorem*{corollary*}{Corollary}

\theoremstyle{definition}
\newtheorem*{definition*}{Definition}
\newtheorem*{remark*}{Remark}
 
\begin{document}

\title{When Periodicity Fails to Guarantee the Existence of Rotation:\\ A Counterexample on the 3-torus  with a Nilpotent Linearization}

\author{W. Oukil\thanks{Corresponding author. Email: \texttt{oukil.walid@gmail.com}} \\
\small\text{Faculty of Mathematics.}\\
\small\text{University of Science and Technology Houari Boumediene.}\\
\small\text{ BP 32 EL ALIA 16111 Bab Ezzouar, Algiers, Algeria.}}

\date{April 14, 2025}

\maketitle
\begin{abstract}
In this manuscript, we construct an explicit counterexample of  smooth infinitely differentiable, periodic dynamical system on the 3-torus for   which the rotation vector exists in the weak sense but fails to exist in the strong sense of bounded deviation. The construction uses Liouville-type arithmetic resonances and shows that periodicity and infinite differentiability alone do not guarantee a bounded drift deviation, even for integrable flows.

To the best of our knowledge, this is the first example of a smooth integrable torus flow with unbounded rotational deviation whose Jacobian is everywhere nilpotent.: the Jacobian matrix is strictly nilpotent, all its eigenvalues are identically zero, and consequently no local exponential stretching or contraction occurs anywhere. The unbounded deviation from the linear drift is generated purely by the accumulation of infinitely many incommensurable frequencies, without any amplification mechanism. This purely neutral local dynamics makes the example particularly relevant for coupled phase oscillator models, where the linearized dynamics around a synchronized state is typically nilpotent or neutral.

Moreover, this work extends the theory of trigonometric polynomial fields. It was previously shown that when the frequency spectrum $\Lambda_f$ is finite, a strong rotation vector always exists. The present counterexample demonstrates that as soon as the spectrum becomes infinite (while retaining full regularity) the strong rotation vector can disappear. Thus, the finiteness of the spectrum cannot be replaced by smoothness alone.
\end{abstract}


\begin{MSC}
34D05, 37B65, 34C15, 37C10.
\end{MSC}

\section{Introduction}

For vector fields on the torus defined by trigonometric polynomials, the existence of the rotation vector in the strong sense is well established \cite{Oukil2023}. In this setting, trajectories admit a linear asymptotic drift with uniformly bounded deviation.

For general smooth periodic vector fields, however, the situation is more subtle.  Classical results on almost periodic differential equations identify regimes where rotation-like behavior holds \cite{Fink1974,Saito1971_01}, but they do not exclude arithmetic effects.

A smooth two-dimensional counterexample was recently constructed in \cite{BrianeHerve2023}, showing that \(C^{\infty}\) regularity alone does not guarantee bounded deviation. The present work  originates from the finite-frequency framework \(\Lambda_f\) introduced in \cite{Oukil2023}. We extend the analysis to an infinite, yet rapidly decaying, frequency spectrum and construct an explicit integrable example on \(\mathbb{T}^3\) where the weak rotation vector exists, but the deviation from the linear drift is unbounded due to Liouville-type resonances.

The dimension three is used here because the first two components generate a Liouville resonance in the third component. More importantly, the vector field possesses a strictly nilpotent linearization: all eigenvalues of the Jacobian are zero, and no local exponential stretching or contraction occurs. The pathological behavior is therefore caused purely by the accumulation of incommensurable frequencies, without any amplification mechanism. This structural feature distinguishes our example from the two-dimensional one and makes it especially relevant for phase models such as the Winfree and Kuramoto systems, where the linearized dynamics around a synchronized state is typically neutral or nilpotent \cite{AriaratnamStrogatz,kuramoto1, WinfreeModel}.

The paper is organized as follows. Section~2 recalls the finite-frequency theory and formulates the central question. Section~3 fixes the definitions of weak and strong rotation vectors. Section~4 contains the explicit construction of the three-dimensional counterexample, the proof of its \(C^{\infty}\) regularity, the computation of the weak rotation vector, and the proof that the strong rotation vector does not exist. Section~5 is devoted to the structural analysis of the counterexample, emphasizing the nilpotent linearization. Section~6 concludes the paper and formulates several open questions.

\section{Motivation and Background: From Finite to Infinite Spectrum}

The starting point of the present work is the theory of \emph{trigonometric polynomial vector fields} developed in \cite{Oukil2023}. A trigonometric polynomial on \(\mathbb{R}^n\) is a finite linear combination of exponentials \(e^{i\langle x,p\rangle}\) with \(p\in 2\pi\mathbb{Z}^n\). For such a field
\[
f(x)=\sum_{p\in\Lambda_f} a_p[f]\, e^{i\langle x,p\rangle},\qquad \Lambda_f\text{ finite},
\]
it was proved in \cite{Oukil2023} that every solution \(x(t)\) of \(\dot{x}=f(x)\) admits a \emph{strong rotation vector} \(\rho\in\mathbb{R}^n\): the limit
\[
\rho=\lim_{t\to\infty}\frac{x(t)}{t}
\]
exists, and moreover the deviation is uniformly bounded,
\[
\sup_{t\geq 0}\|x(t)-\rho t\|<\infty.
\]

The proof relies in an essential way on the finiteness of the frequency set \(\Lambda_f\). When the spectrum is finite, all Fourier series involved are actually finite sums, and the Diophantine conditions needed to invert small divisors can be satisfied without loss of regularity. In this finite-dimensional frequency setting, the flow always behaves asymptotically like a linear drift plus a bounded almost-periodic oscillation.

A natural question immediately arises: is the strong rotation property a consequence of the \(C^{\infty}\) regularity of the field, or of the finiteness of its spectrum? In other words, if we allow an \emph{infinite} Fourier spectrum while keeping the field smooth (i.e., with rapidly decaying coefficients), does the strong rotation vector persist? 

\section{Definitions and Framework}

Let $\mathbb{T}^3 = \mathbb{R}^3/\mathbb{Z}^3$ denote the 3-dimensional torus. For a flow $\varphi^t: \mathbb{T}^3 \to \mathbb{T}^3$, consider a lift $X: \mathbb{R} \to \mathbb{R}^3$ satisfying $\pi(X(t)) = \varphi^t(\pi(X(0)))$ where $\pi: \mathbb{R}^3 \to \mathbb{T}^3$ is the canonical projection.

\begin{definition}[Rotation vectors]
1. The {rotation vector exists in the weak sense} if
\[
\rho \coloneqq \lim_{t\to \infty}\frac{X(t) - X(0)}{t}
\]
exists in $\mathbb{R}^3$.

2. The {rotation vector exists in the strong sense} if $\rho$ exists and moreover
\[
\sup_{t\geq 0}\| X(t) - X(0) - \rho t\| < \infty.
\]
\end{definition}

\section{The counterexample}

\subsection{Construction and regularity}

Let $r \in \mathbb{R}$ be a Liouville number. By definition, for each $n \geq 1$ there exist integers $p_n \geq 2$ and $q_n \geq 1$ such that
\[
0 < |r p_n - q_n| < p_n^{-n }.
\] 
Using the fact that $p_n\geq2$, fix a sequence $(p_m, q_m)_{m\in \mathbb{N}^*}$ such that
\begin{equation}\label{eq:liouville-approx}
0 < |r p_m - q_m| < p_m^{-m }<|r p_{m-1} - q_{m-1}|,\quad \forall m\geq1.
\end{equation} 
Define the vector field $h: \mathbb{T}^3 \to \mathbb{R}^3$ by
\[
h(z) = \left(r,\ 1,\ \sum_{m=1}^{\infty} \, m\, p_m^{-m+1} \, \cos(2\pi(z_1 p_m - z_2 q_m))\right).
\]

\begin{lemma}[Smoothness]\label{lem:smooth}
The vector field $h$ belongs to $C^{\infty}(\mathbb{T}^3)$.
\end{lemma}

\begin{proof}
For each $k \geq 0$, we estimate the $k$-th derivative of the series termwise:
\[
\sum_{m=1}^{\infty} m\, p_m^{-m+1} \, (2\pi\|(p_m, q_m)\|)^k 
\leq (2\pi)^k \sum_{m=1}^{\infty} m\, p_m^{-m+1} \max\{ p_m^k ,q_m^k \}.
\]
From the Liouville condition we have $|r p_m - q_m| < p_m^{-m }$, which implies $q_m \le |r| \, p_m + p_m^{-m }$. Since $p_m \geq 2$, the term $p_m^{-m }$ is bounded and decreases as $m \to \infty$. For a fixed $k$ and large $m$, $q_m$ is dominated by the linear term $|r|p_m$. More precisely,
\[
q_m^k < (|r|p_m + p_m^{-m })^k = (|r| + p_m^{-m-1})^k \, p_m^k \le (|r| + 1)^k \, p_m^k,\quad m\geq1.
\]
Thus,
\[
m\, p_m^{-m+1} \max\{ p_m^k ,q_m^k \} \leq m\, (|r| + 1)^k \, p_m^{-m+1+k}.
\]
Since $p_m \geq 2$, we have $p_m^{-m+1+k} \le 2^{-m+1+k}$ for all $m \ge k+2$. Hence for $m \ge k+2$,
\[
m\, p_m^{-m+1} \max\{ p_m^k ,q_m^k \} \le (|r|+1)^k \, 2^{k+1} \, m \, 2^{-m}.
\]
The series $\sum m\, 2^{-m}$ converges, so by the Weierstrass M-test, the series converges absolutely in $C^k(\mathbb{T}^3)$ for every $k$, implying $h \in C^{\infty}(\mathbb{T}^3)$.
\end{proof}

Consider the initial value problem
\begin{equation}\label{eq:system}
\dot{x} = h(x), \qquad x(0) = 0.
\end{equation}
The system is integrable with explicit solution
\begin{equation}\label{phi1}
x_1(t) = rt, \quad x_2(t) = t, \quad 
x_3(t) = \int_0^t \sum_{m=1}^{\infty} m\, p_m^{-m+1} \cos(2\pi(r p_m - q_m)s)\, ds.
\end{equation}

\begin{lemma}[Termwise integration]\label{lem:integration}
For each $t \geq 0$,
\[
x_3(t) = \sum_{m=1}^{\infty} \frac{m\, p_m^{-m+1}}{2\pi (r p_m - q_m)} \sin(2\pi(r p_m - q_m)t).
\]
\end{lemma}

\begin{proof}
Since $p_m \geq 2$, the series $\sum_{m=1}^{\infty} m\, p_m^{-m+1}$ converges. By the Weierstrass M-test, for every fixed $t>0$ the series converges uniformly on $[0,t]$. Hence integration and summation commute. We obtain
\begin{align*} 
x_3(t) &= \int_0^t \sum_{m=1}^{\infty} m\, p_m^{-m+1} \cos(2\pi(r p_m - q_m)s)\, ds\\
&= \sum_{m=1}^{\infty} m\, p_m^{-m+1} \int_0^t \cos(2\pi(r p_m - q_m)s)\, ds.
\end{align*}
By \eqref{eq:liouville-approx} we have $|r p_m - q_m| > 0$, therefore
\[
\int_0^t \cos(2\pi(r p_m - q_m)s)\, ds = \frac{1}{2\pi}\frac{1}{r p_m - q_m} \sin(2\pi(r p_m - q_m)t).
\]
\end{proof}

\subsection{Weak rotation vector}

\begin{proposition}[Weak rotation vector]\label{prop:weak}
The rotation vector of \eqref{eq:system} exists in the weak sense and equals $\rho = (r,1,0)$.
\end{proposition}

\begin{proof}
From \eqref{phi1} we have $x_1(t)/t = r$ and $x_2(t)/t = 1$ exactly. For the third component,
\[
x_3(t) = \int_0^t \sum_{m\geq 1} m\, p_m^{-m+1} \cos\left(2\pi(r p_m - q_m)s\right) ds.
\]
Let $N \in \mathbb{N}^*$. Split the sum:
\begin{align*}
x_3(t) &= \sum_{m=1}^{N} m\, p_m^{-m+1} \int_0^t \cos\left(2\pi(r p_m - q_m)s\right) ds \\
&\quad + \int_0^t \sum_{m> N} m\, p_m^{-m+1} \cos\left(2\pi(r p_m - q_m)s\right) ds \\
&=: I_N(t) + J_N(t).
\end{align*}
For $I_N(t)$ we use the explicit integration and the boundedness of $\sin$:
\[
|I_N(t)| \le \sum_{m=1}^{N} \frac{m\, p_m^{-m+1}}{2\pi\,|r p_m - q_m|},
\]
which is independent of $t$. Hence $\lim_{t\to\infty} t^{-1} I_N(t) = 0$.  
For $J_N(t)$ we estimate directly:
\[
|J_N(t)| \le \int_0^t \sum_{m> N} m\, p_m^{-m+1} \, ds = t \sum_{m> N} m\, p_m^{-m+1}.
\]
Thus
\[
\limsup_{t\to\infty} \frac{|x_3(t)|}{t} \le \sum_{m> N} m\, p_m^{-m+1}.
\]
Because $\sum_{m=1}^\infty m\, p_m^{-m+1}$ converges, the right-hand side tends to $0$ as $N\to\infty$. Consequently, $\lim_{t\to\infty} x_3(t)/t = 0$.  
The weak rotation vector is therefore $(r,1,0)$.
\end{proof}

\subsection{Failure of strong rotation vector}

\begin{proposition}[Unbounded deviation]\label{prop:unbounded}
The function $t \mapsto x_3(t)$ is unbounded. Consequently, the rotation vector does not exist in the strong sense.
\end{proposition}

\begin{proof}
We argue by contradiction. Suppose that there exists a finite constant $M$ such that
\[
\sup_{t\geq 0} |x_3(t)| \le M.
\]
For every $m \ge 1$ set $\lambda_m := r p_m - q_m$. By \eqref{eq:liouville-approx} we have $|\lambda_m| \neq 0$ and the sequence $(|\lambda_m|)$ is strictly decreasing.

Fix an integer $n \ge 1$ and consider the time average
\[
A_n(t) := \frac{1}{t}\int_0^t x_3(s)\,\sin(2\pi \lambda_n s)\,ds.
\]
The boundedness assumption implies $|A_n(t)| \le M$ for all $t>0$. We will compute the limit of $A_n(t)$ as $t\to\infty$ using integration by parts.

Let $u = x_3(s)$, $dv = \sin(2\pi \lambda_n s)\,ds$. Then $du = \dot{x}_3(s)\,ds$ and $v = -\frac{\cos(2\pi \lambda_n s)}{2\pi \lambda_n}$. Hence
\begin{align*}
\int_0^t x_3(s)\sin(2\pi \lambda_n s)\,ds
&= -\frac{x_3(t)\cos(2\pi \lambda_n t)}{2\pi \lambda_n} + \frac{x_3(0)\cos(0)}{2\pi \lambda_n} \\
&\quad + \frac{1}{2\pi \lambda_n}\int_0^t \dot{x}_3(s)\cos(2\pi \lambda_n s)\,ds .
\end{align*}
Since $x_3(0)=0$ and $|x_3(t)|\le M$, the boundary term is $O(1)$ and therefore
\[
\lim_{t\to\infty} \frac{1}{t}\Bigl( -\frac{x_3(t)\cos(2\pi \lambda_n t)}{2\pi \lambda_n} \Bigr) = 0 .
\]
Consequently,
\begin{equation}\label{eq:limit-av}
\lim_{t\to\infty} A_n(t) = \frac{1}{2\pi \lambda_n} \lim_{t\to\infty} \frac{1}{t}\int_0^t \dot{x}_3(s)\cos(2\pi \lambda_n s)\,ds .
\end{equation}

From the differential equation \eqref{eq:system} and the construction,
\[
\dot{x}_3(s) = \sum_{m=1}^\infty m\, p_m^{-m+1} \cos(2\pi \lambda_m s),
\]
where the series converges absolutely and uniformly on $\mathbb{R}$ because $\sum m\, p_m^{-m+1} < \infty$.

Insert this expansion into \eqref{eq:limit-av}. Since

\[
\sum_{m=1}^{\infty} m\,p_m^{-m+1}<\infty
\]

and

\[
\left|
\frac1t
\int_0^t
\cos(2\pi\lambda_n s)
\cos(2\pi\lambda_m s)\,ds
\right|
\le 1,
\]

the dominated convergence theorem for series allows us to interchange
the limit and the summation. Hence

\begin{align*}
\lim_{t\to\infty} A_n(t)
&=
\frac{1}{2\pi\lambda_n}
\sum_{m=1}^{\infty}
m\,p_m^{-m+1}
\lim_{t\to\infty}
\frac1t
\int_0^t
\cos(2\pi\lambda_n s)
\cos(2\pi\lambda_m s)\,ds .
\end{align*}
For each $m$, the product of cosines satisfies
\[
\cos(2\pi \lambda_n s)\cos(2\pi \lambda_m s) = \frac{1}{2}\Bigl[ \cos\bigl(2\pi(\lambda_m-\lambda_n)s\bigr) + \cos\bigl(2\pi(\lambda_m+\lambda_n)s\bigr) \Bigr].
\]
If $m = n$, the term $\cos(0)=1$ gives a non‑zero average:
\[
\frac{1}{t}\int_0^t \cos^2(2\pi \lambda_n s)\,ds = \frac12 + O\Bigl(\frac1t\Bigr) \;\longrightarrow\; \frac12 .
\]
If $m \neq n$, the frequencies $\lambda_m \pm \lambda_n$ are non‑zero (the sequence $|\lambda_m|$ is strictly decreasing), hence
\[
\int_0^t \cos(2\pi(\lambda_m\pm\lambda_n)s)\,ds = \frac{\sin(2\pi(\lambda_m\pm\lambda_n)t)}{2\pi(\lambda_m\pm\lambda_n)}
\]
remains bounded uniformly in $t$, and the corresponding contribution divided by $t$ tends to $0$.

Because the series $\sum m\, p_m^{-m+1}$ converges absolutely, we can pass to the limit term by term, obtaining
\[
\lim_{t\to\infty} A_n(t) = \frac{1}{2\pi \lambda_n} \cdot n\, p_n^{-n+1} \cdot \frac12 = \frac{n\, p_n^{-n+1}}{4\pi \lambda_n}.
\]
Taking absolute values and using $|A_n(t)| \le M$,
\[
\frac{n\, p_n^{-n+1}}{4\pi\,|\lambda_n|} \le M \qquad \text{for every } n \ge 1.
\]
The Liouville condition \eqref{eq:liouville-approx} gives $|\lambda_n| < p_n^{-n}$, hence $1/|\lambda_n| > p_n^{n}$. Therefore
\[
\frac{n\, p_n^{-n+1}}{4\pi\,|\lambda_n|} > \frac{n\, p_n^{-n+1} \cdot p_n^{n}}{4\pi} = \frac{n\, p_n}{4\pi} \ge \frac{n\cdot 2}{4\pi} = \frac{n}{2\pi}.
\]
Thus we obtain $n/(2\pi) < M$ for all $n \ge 1$, an impossibility. The contradiction shows that $x_3(t)$ cannot be bounded, and therefore the rotation vector does not exist in the strong sense.
\end{proof}

\section{Structural analysis: nilpotent linearization}

A key feature of this counterexample is the structure of its linearization. The Jacobian matrix of \(h\) at any point \(z=(z_1,z_2,z_3)\) is
\[
Dh(z) = \begin{pmatrix}
0 & 0 & 0 \\
0 & 0 & 0 \\
\partial_{z_1}h_3 & \partial_{z_2}h_3 & 0
\end{pmatrix}.
\]
All its eigenvalues are exactly zero; the matrix is strictly nilpotent. Hence, there is no local exponential stretching or contraction of nearby trajectories: the dynamics is completely neutral from the viewpoint of linear stability analysis.

Such a purely neutral setting is especially relevant for phase models of coupled oscillators, such as the Winfree and Kuramoto systems, where the coupling typically involves only the phases and the linearization around a synchronized state is nilpotent or has zero eigenvalues.

\section{Conclusion and questions}

We have constructed an explicit \(C^\infty\) periodic vector field on \(\mathbb{T}^3\) whose flow admits a weak rotation vector \(\rho=(r,1,0)\) but fails to possess a strong rotation vector: the deviation \(x_3(t)\) is unbounded. The linearization of this field is strictly nilpotent (all eigenvalues are zero), so no local exponential stretching or contraction occurs. Thus the unbounded deviation does not rely on any linear amplification mechanism.

A smooth counterexample with a Liouville slope was already known in dimension two. The present three‑dimensional example is therefore not the first \(C^\infty\) counterexample; its novelty lies in two features. First, the nilpotent structure shows that the loss of the strong rotation vector can happen in a completely neutral local environment, a situation natural for phase oscillator models. Second, the construction directly extends the finite‑frequency framework of \cite{Oukil2023}: that paper proved that a trigonometric polynomial field (finite spectrum) always admits a strong rotation vector. The example demonstrates that the finiteness assumption cannot be replaced by \(C^\infty\) regularity alone, because an infinite but rapidly decaying spectrum suffices to destroy the strong rotation property.

Several natural questions remain open:

\begin{enumerate}
\item \textbf{Arithmetic conditions:} Characterize the Diophantine properties of frequency vectors that guarantee bounded deviation for smooth (or analytic) periodic systems. Is a Diophantine condition necessary and sufficient?

\item \textbf{Regularity versus decay:} How does the required decay rate of Fourier coefficients depend on the arithmetic type of the frequencies to ensure bounded deviation?

\item \textbf{Non-integrable perturbations:} Does the phenomenon persist under small non-integrable perturbations? What happens in nearly-integrable Hamiltonian systems?

\item \textbf{Higher-dimensional generalisations:} Can similar counterexamples be constructed on \(\mathbb{T}^n\) for \(n>3\) with rotation sets instead of vectors?
\end{enumerate}

These questions highlight the interplay between arithmetic, regularity, and asymptotic dynamics in periodic systems, an interplay that is often implicit in applied studies of oscillatory networks.

\end{document}